\documentclass[10pt,a4paper,reqno]{amsart}
\usepackage{amsfonts,amsthm,latexsym,amsmath,amssymb,amscd,amsmath,
epsf, showlabels,tikz
}
\usepackage{graphicx}

\newtheorem{theo+}           {Theorem}
\newtheorem{prop+}           {Proposition}
\newtheorem{coro+}           {Corollary}
\newtheorem{lemm+}           {Lemma}

\theoremstyle{definition}
\newtheorem{defi+}           {Definition}
\newtheorem{problem}         {Problem}

\theoremstyle{remark}
\newtheorem{rema+}           {Remark}

\newenvironment{theorem}{\begin{theo+}}{\end{theo+}}
\newenvironment{proposition}{\begin{prop+}}{\end{prop+}}
\newenvironment{corollary}{\begin{coro+}}{\end{coro+}}
\newenvironment{lemma}{\begin{lemm+}}{\end{lemm+}}
\newenvironment{remark}{\begin{rema+}}{\end{rema+}}
\newenvironment{definition}{\begin{defi+}}{\end{defi+}}

\newtheorem{THEO}{Theorem}

\newtheorem{LEM}{Lemma}

\newtheorem{PROP}{Proposition}

\newcommand{\al}{\alpha}

\newcommand {\Ga} {\Gamma}

\newcommand{\bC}{\mathbb C}
\newcommand{\bR}{\mathbb R}

\newcommand{\D}{\mathcal D}
\newcommand{\C}{\mathcal C}

\newcommand{\QQ}{\mathbb{Q}}

\def\d{{\partial}}

\def\newop#1{\expandafter\def\csname #1\endcsname{\mathop{\rm
#1}\nolimits}}

\newop{slc}
\newop{sign}
\newop{mdeg}

\begin{document}
          \numberwithin{equation}{section}

          \title[On existence of quasi-Strebel structures for  $k$-differentials ]{On existence of quasi-Strebel structures for meromorphic $k$-differentials }

\author[B.~Shapiro]{Boris Shapiro}
\address{Department of Mathematics, Stockholm University, SE-106 91
Stockholm,
         Sweden}
\email{shapiro@math.su.se}

          \author[G.~Tahar]{Guillaume Tahar}
\address{Faculty of Mathematics and Computer Science, Weizmann Institute of Science, Rehovot, 7610001 Israel}
\email{tahar.guillaume@weizmann.ac.il}

\date{\today}

\begin{abstract}   In this paper, motivated by the classical notion of a Strebel quadratic differential on a compact Riemann surfaces without boundary 
 we introduce the notion of a quasi-Strebel structure for a meromorphic differential of an arbitrary order.  It turns out that every differential of even order $k\ge 4$ satisfying certain natural conditions at its singular points admits such a structure. The case of differentials of odd order is quite different and our existence result involves some arithmetic conditions.
  We discuss the set of quasi-Stebel structures associated to a given differential and introduce the subclass of positive  $k$-differentials. Finally, we provide a  family of examples of positive rational differentials  and explain their connection  with the classical Heine-Stieltjes theory of linear differential equations with polynomial coefficients. 
\end{abstract}

\maketitle

\section {Introduction}
 
 \begin{defi+}
A (meromorphic) differential $\Psi$ of order $k\ge 2$ on a compact orientable Riemann surface $Y$ without boundary is
a (meromorphic) section of the $k$-th tensor power 
$(T^*_{\mathbb{C}}Y)^{\otimes k}$ of the holomorphic cotangent
bundle $T^*_{\mathbb{C}}Y$. The zeros and the poles of $\Psi$
constitute the set $Cr_\Psi$ of \emph{critical points} (alias, singularities) of $\Psi$. 
\end{defi+}

In what follows, we will use the convention that the order of a zero is a positive integer while the order of a pole is a negative integer. 
For a differential $\Psi$ of order $k$ given locally by $f(z)dz^k$ in a neighbourhood of a non-critical point,  we define $k$ locally distinct direction fields called {\em horizontal} which are given by the condition that $f(z)dz^k$ is positive. Integral curves of these direction fields are called (horizontal) \emph{trajectories}. Any two  consecutive direction fields differ by the angle $\frac{2\pi}{k}$   and for even $k=2\ell$, they form $\ell$ locally distinct line fields. However, for a general differential of order $k>2,$ it is impossible to globally distinguish these $k$  direction fields because every two of them are obtained by the analytic continuation of each other.  Multiplication of $\Psi$ by the scalar factor $e^{\frac {\pi i }{k}}, \; i=\sqrt{-1}$ does not change these direction fields and their trajectories.   

Extracting the $k$-th root, one can globally interprete an arbitrary differential of order $k$ as a multi-valued meromorphic 
abelian differential on $Y$.   
Outside of the poles and the zeroes of $\Psi$, it is locally representable as the $k$-th power of a holomorphic abelian differential $\omega$. Integration of $\omega$ in a neighborhood of a point $z_{0}$ gives local coordinates whose transition maps are of the type $z \mapsto e^{i\theta}z+c$ where $\theta\in\left\lbrace 0,\dfrac{2\pi}{k}\dots,(k-1)\dfrac{2\pi}{k} \right\rbrace$. 
 In such a way, the Riemann surface $Y$ punctured at its singularities is locally isometric to the complex plane $\mathbb{C}$ endowed with the flat structure induced by $dz$. In other words, we obtain a flat metric on $Y\setminus Cr_\Psi$ induced by the integration of $\omega$ where the distance between two points $p_1$ and $p_2$ is $\vert \int_{p_1}^{p_2}\omega \vert $, see \cite{ BCGGM, Ta}. (In particular, $dz$ induces the usual Euclidean metric on $\bC$.)  The surface $Y\setminus Cr_\Psi$ with the latter flat metric will be called the \emph{flat model} of $\Psi$.
 
 \smallskip
 This flat metric extends from $Y\setminus Cr_\Psi$ to all zeros and poles of order $\ell >  -k$; such  singularities are called \emph{conical of order $\ell$}. Any conical singularity of order $\ell$ has a neighborhood on $Y$ isometric to a neighborhood of the apex of a flat cone with the total angle  $\frac{(\ell+k)2\pi}{k}$. Analogously, a pole of order $\ell \le -k$ corresponds to the point at $\infty$ of a flat cone with the total angle  $-\frac{(\ell+k)2\pi}{k}$, see Proposition~\ref{pr:local}. By convention, flat cones of angle $0$ (i.e., poles of order $-k$) are half-infinite cylinders.  
 

\smallskip
 In general, differentials of order $k>2$ are much less studied than their quadratic counterpart. Nevertheless they sometimes appear in the literature, see e.g.  \cite{Str2, BCGGM, BH, DW, HLL, La, PePi, Ta, Bu} and, as we already mentioned,  are closely related to the theory of locally flat surfaces which is used  in the study of the moduli spaces of curves and  dynamics of interval exchange maps, see e.g. \cite{Zo}. Moreover, theory of translation surfaces (which correspond to the abelian differentials) provides new important insights into the dynamics of billiards by means of  algebraic geometry and renormalization theory.

\smallskip
Moduli spaces of pairs $(Y,\Psi)$ consisting of  a compact Riemann surface $Y$ and a $k$-differential on it are naturally stratified by the loci where the differentials have the same set of multiplicities of their singular points. For example, in the case of genus zero, a meromorphic differential of order $k$ can be globally represented as $\Psi(z)=R(z)dz^k$,  where $R(z)$ is a rational function. So, up to a non-vanishing constant factor, $\Psi$ is completely determined by the multiplicities and the positions of its singularities. In particular, the strata of $k$-differentials can be identified with certain configuration spaces of divisors on $\bC P^1$ and are connected, see \cite{Bo}.

A natural system of coordinates on the latter strata is given by the periods of $\root k \of \Psi$ with respect to the relative homology of $Y$ punctured at all poles of order smaller than or equal to $-k$ taken with respect to the set of all conical singularities (i.e. all zeroes together with all poles of order exceeding  $-k$). In the genus zero case, projective automorphisms allow us to assign positions of three singularities to $(0,1,\infty)$ and strata of $k$-differentials with $m$ singularities on the sphere are complex-analytic orbifolds of dimension of $m-2$. More details can be found in \cite{BCGGM}.

Local classification of singularities of differentials of order $k>2$  under the action of the group of local biholomorphisms is quite similar to that of quadratic differentials and can be found in Proposition~\ref{pr:local} below. 
   
   \medskip
  The main motivation of  this paper is as follows.  
Recall that  a quadratic differential on  a compact Riemann surface $Y$ is called \emph{Strebel}  if almost all its horizontal trajectories are closed. Such a phenomenon can never happen for a $k$-differential $\Psi$ of order $k\ge 3$ unless it is a power of a $1$-form or a quadratic differential, see Corollary~\ref{cor:quadr} below. Thus for $k\ge 3$,  in order to define something similar to a Strebel differential one should instead of using only smooth closed horizontal trajectories of $\Psi$ (which might exist in some domains of $Y$) allow a certain type of closed broken trajectories  and try to cover $Y$ with those. However, to avoid trivialities, one should impose substantial restrictions on the set where broken trajectories can switch from one smooth piece to another. An attempt to develop an appropriate notion is carried out below by introducing quasi-Strebel structures,  see Definition 10.
 
 \medskip
 The purpose of this paper is to develop elements of a general  theory of $k$-differentials which as much as possible resembles that of quadratic differentials and  to point out the connection of a certain class of $k$-differentials with the classical Heine-Stieltjes theory. It seems very plausible that special types of $k$-differentials should naturally appear in the asymptotic analysis of linear ODE of high order, comp. e.g. \cite{AKT}.   In what follows, theorems, conjectures, etc., labeled by letters, are borrowed from the existing literature, while those labeled by numbers are hopefully new.
 
  \medskip
 Our main result  (Theorem 1) claims the existence of a quasi-Strebel structure for every meromorphic $k$-differential with  admissible singularities if $k$ is even. A similar statement for  odd $k$, currently requires an additional (and very restrictive)  hypothesis of rationality of the periods of a differential. On the other hand, we strongly believe that for $k$ odd, there are arithmetic obstructions to the existence of a quasi-Strebel structure. 
 
  \smallskip
 The  paper is organized as follows. Section \S~\ref{sec:basic} contains main definitions and results related to $k$-differentials and quasi-Strebel structures. 
 Section \S~\ref{sec:HS} introduces a special class of  positive differentials and provides a short explanation how such differentials appear in the Heine-Stieltjes theory. Finally, in \S~\ref{sec:final} we present a number of relevant open questions.

\medskip
\noindent
\emph {Acknowledgements.} The first author is indebted to his friend and coauthor  Professor Y.~Baryshnikov of the University of Illinois at Urbana-Champaign for numerous discussions of this topic. He also wants to acknowledge the financial support of his research provided by the Swedish Research Council grant  2016-04416. 
The second author is sincerely grateful to the department of mathematics, Stockholm University for the hospitality in August and October 2018 as well as February 2020 when this project was carried out. He also acknowledges the support of the Israel Science Foundation (grant No 1167/17).

\section{Basic notions and the main theorem}\label{sec:basic}

\subsection{Trivium on $k$-differentials}
The next definitions are straightforward generalizations of those for quadratic differentials. 

\begin{defi+} The {\em canonical length element} on $Y$ associated with a $k$-differential $\Psi$ given locally as $\Psi=f(z) dz^k$ is defined by
$$|dw|=|f(z)|^{\frac{1}{k}} |dz|.$$
\end{defi+}

(Notice that the associated area element on $Y$ equals $|f(z)|^{\frac{2}{k}}dxdy$ where $z=x+iy$. Both $|dw|$ and the area element are globally well-defined.)

\begin{defi+}
a) The {\em norm} of a $k$-differential  $\Psi=f(z) dz^k$ is defined as
$$\iint_\Ga |f(z)|^{\frac{2}{k}}dxdy.$$
\noindent
(Again $|f(z)|^{\frac{2}{k}}dxdy$ is a globally well-defined $2$-form on $Y$.)

\noindent
b) If $\Psi=f(z) dz^k$  is meromorphic, then its singular point $p$ is called {\em finite or conical} if $p$ has a neighbourhood with a finite associated area.
\end{defi+}

\begin{rema+} Finite critical points of $\Psi$ are exactly its zeros and poles of order greater than $-k$ (i.e. conical singularities) and a meromorphic $k$-differential $\Psi$ has a finite norm if and only if all its critical points are finite.
 \end{rema+}

\begin{defi+}[comp. Definition 20.1 in \cite {Str}]  A trajectory of $\Psi$ is called {\em critical} if it starts or ends at a conical singularity.
\end{defi+}

\begin{defi+}
The {\em distinguished} or {\em natural} parameter $W$ associated with a $k$-differential $\Psi=f(z) dz^k$ is defined by
$$W=\int\root k \of {f(z)} dz$$
 for some branch of the $k$-th root.
\end{defi+}

\begin{rema+}
Obviously, a sufficiently small neighbourhood of any regular point $p$  is homeomorphically mapped by a branch of $W$ onto an open set in the $W$-plane. Notice that
$dW=\root k \of {f(z)} dz$ which implies $(dW)^k=f(z) dz^k$.
\end{rema+}

A standard basic result about $k$-differentials  which can be found in several sources is as follows. 

\begin{LEM} \label{lm:eulerk} The Euler characteristic of $(T_\bC^*Y)^{\otimes k}$ equals $\chi(Y) k$, where $\chi(Y)$ is the Euler characteristic of the underlying curve $Y$. Therefore, the difference between the number of poles and zeros (counted with multiplicity) of a meromorphic  differential $\Psi$ of order $k$ on $Y$ equals $\chi(Y) k$. In particular, the number of poles minus the number of zeros of any rational $k$-differential $\Psi$   equals $2k$.
\end{LEM}

Next, following closely \S~6 of Ch.~3 in \cite {Str}, we provide normal forms for a $k$-differential near its critical points. (These results can be found in Proposition~3.1 of \cite{BCGGM} and were independently obtained about 10 years ago by the first author (unpublished).)

\begin{PROP}\label{pr:local} Given a $k$-differential $\Psi$ on a sufficiently small disk $V$ centered at $0$, denote by $m$ the order 
$\text{ord}_0  \Psi$. Then there exists a conformal map $\xi:(\Delta_R,0)\to (V,0)$ defined on a disk of sufficiently small radius $R$, and a number $r\in \bC$ such that 
$$\xi^*(\Psi)=\begin{cases} z^mdz^k\quad \text{if }m>-k \text{ or } k\not | m,\\
\frac{r}{z^k}dz^k \quad \text{if }m=-k,\\
\left(z^{m/k}+\frac{s}{z}\right)^kdz^k \quad \text{if }m<-k  \text{ and } k|m,
\end{cases}
$$
where in the last case $s\in \bC$ and $s^k=r$. The germ of $\xi$ is unique up to multiplication by a $(m+k)$-th root of unity when $m > -k$ or $k\not | m$, and up to multiplication by a non-zero constant if 
$m = -k$.
\end{PROP}

\begin{figure}

\begin{center}
\includegraphics[scale=0.423]{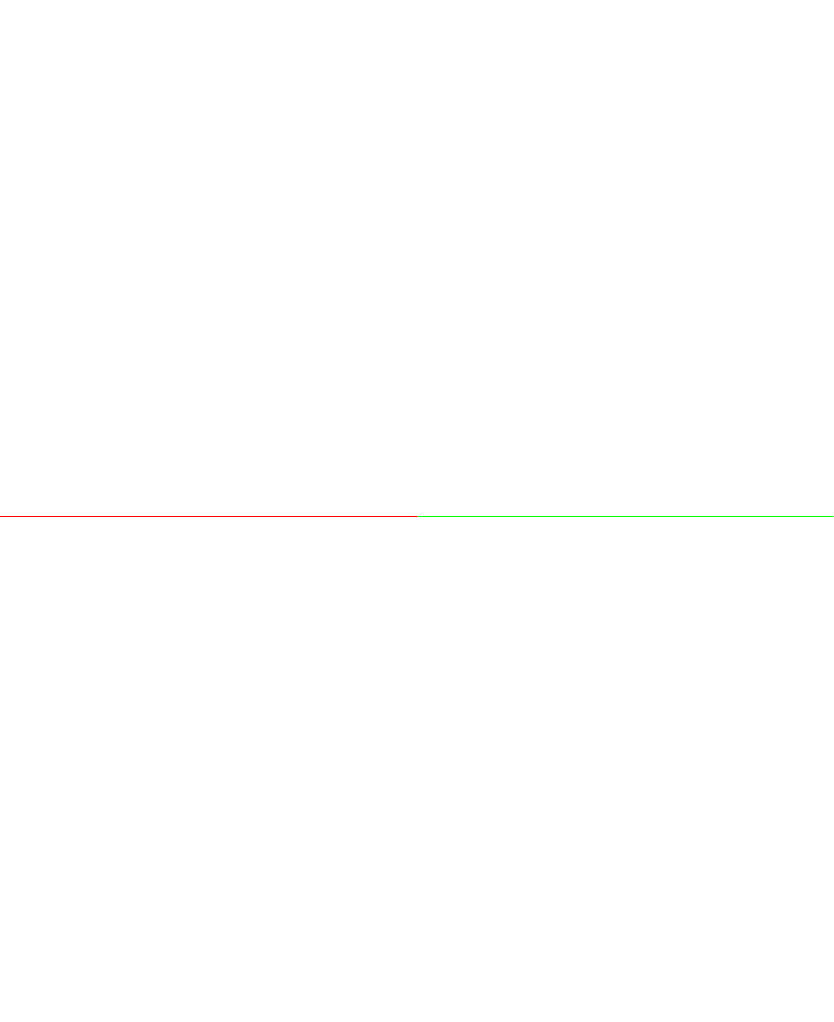} 
\includegraphics[scale=0.57]{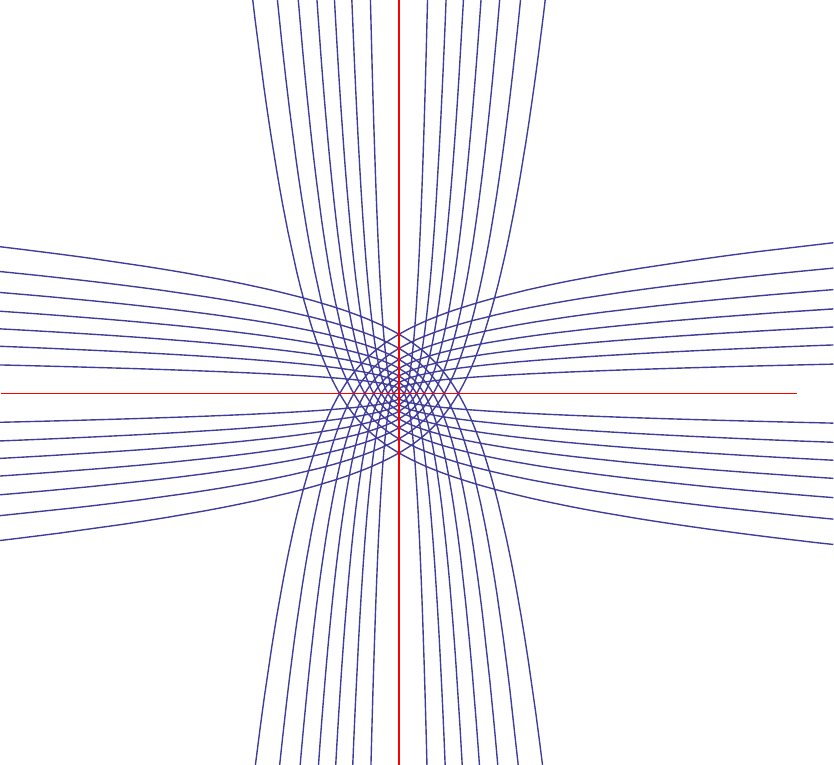} 
\includegraphics[scale=0.45]{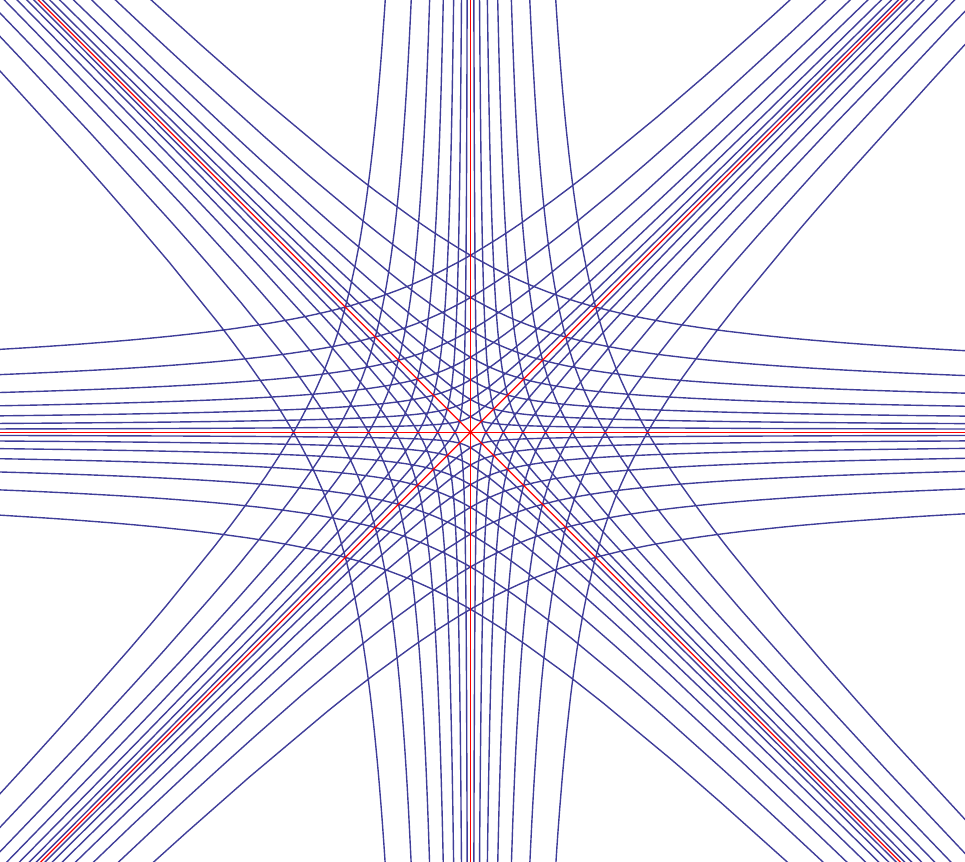}
\includegraphics[scale=0.55]{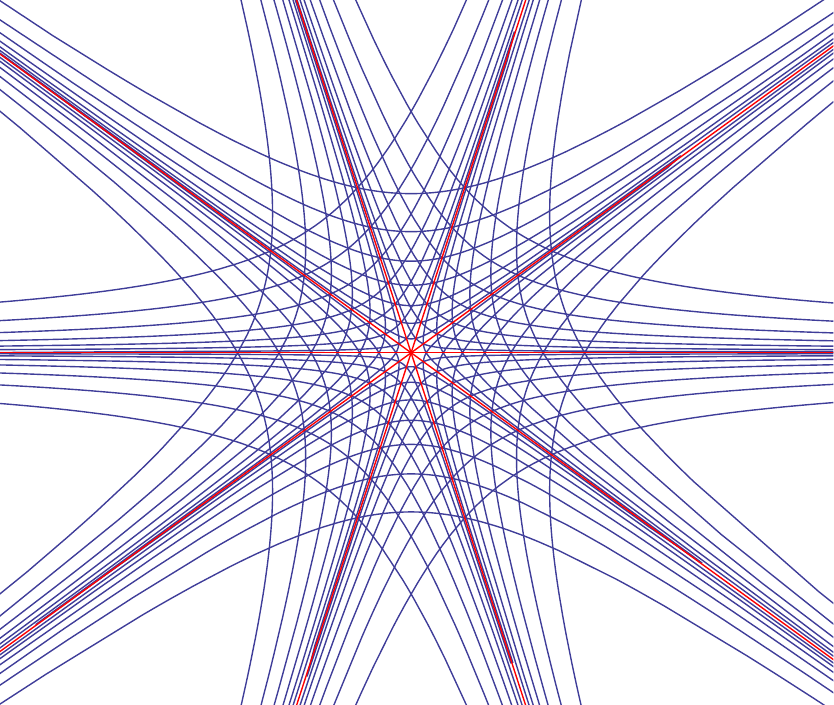}
\end{center}

\vskip 0.5cm

\caption {Local phase portraits of a cubic differential near a simple and double poles (top row) and simple and double zeros (bottom row).}
\label{fig1}
\end{figure}

  The following statement describes possible behavior of non-critical global trajectories, see \cite{Ta}, Proposition 5.5. 
  
  \begin{LEM}\label{lm:tahar}
 Non-critical trajectories of a $k$-differential $\Psi$ are of the following $3$ types: 

\smallskip
\noindent
(i) closed trajectories (with or without selfintersections);

\smallskip
\noindent
(ii) trajectories (with or without selfintersections)  one or both ends of which approach a pole of order smaller than or equal to $-k$; 

\smallskip
\noindent
(iii) trajectories (with or without  selfintersections) which are dense in some polygonal domains in the flat model of $\Psi$.
\end{LEM}

 Observe that the analytic continuation of a branch of $\root k \of \Psi$ of a $k$-differential $\Psi$  along some loop on $Y$ may result in a different branch. The ratio between the values of these two branches is a $k$-th root of the unity called the \emph{holonomy} of the loop. This ratio describes the failure of the identification  of the circle of directions in the parallel transport along the loop. Observe that in particular, the holonomy of a loop around a singularity of order $n$ depends on the congruence class of $n$ modulo $k$. To formulate our next result we need the following notion.

\begin{defi+}
In the above notation,  the \emph {holonomy group of  $Y$} is the subgroup of $G_k\subset U(1)$ generated by the holonomies of all possible  loops on $Y$ where $G_k$ denotes the group of $k$-th roots of unity. Given a domain $\Omega \subseteq Y$, we define the holonomy of $\Omega$ as the subgroup of $G_k$ generated by the holonomies of loops contained in $\Omega$.\end{defi+}

\smallskip
By a  \emph{density domain} we denote the closure of a trajectory in case (iii) of Lemma B without selfintersections. The following claim seems to be new. 

\begin{proposition}\label{pr:density} The holonomy of a density domain is either trivial, i.e. consists of $\{1\}\in G_k$ or equals $\{\pm 1\}\subset G_k$ in which case $k$ is even. 
\end{proposition}

\begin{remark} In the case when the holonomy group of a density domain is trivial the $k$-differential  $\Psi$ restricted to the density domain can be represented as the $k$-th power of a $1$-form.   In the case when the holonomy group of a density domain is $\{\pm 1\}$ then $k$ is even and $\Psi$ restricted to the density domain can be represented as the $\frac{k}{2}$-th power of a quadratic differential.  
\end{remark}

\begin{proof}
Consider a non-selfintersecting global horizontal trajectory $\gamma$  which is not periodic. Its closure $A$ is a subset of the surface $Y$   foliated by horizontal trajectories parallel to $\gamma$. Since $\gamma$ is non-selfintersecting these other trajectories are non-selfintersecting either. Therefore, $A$ is a polygonal domain on the flat model of $Y$ whose boundary is a union of segments of horizontal trajectories parallell to $\gamma$.
Thus $A$ is a flat polygon with some of its sides identified. 

Notice that such presentation is not unique. Nevertheless, the  angle between every side of a polygon and the foliation by the trajectories lines is well-defined. Therefore, the holonomy of any simple loop that crosses   exactly one side exactly once is either $1$ or  $-1$. Such loops generate the fundamental group of $A$ which finishes the proof. 
\end{proof}

\begin{corollary} \label{cor:quadr} If a meromorphic $k$-differential $\Psi$ on a compact Riemann surface $Y$ with $k\ge 3$  has a family of smooth periodic trajectories covering $Y$ almost everywhere then there either exists a $1$-form $\omega$ or a quadratic differential $q$ such that either $\Psi=\omega^k$ or $\Psi=q^{\frac {k}{2}}$. 
\end{corollary}

\begin{proof}
If there is a family of smooth closed trajectories covering $Y$ almost everywhere, then the direction field of $\Psi$ tangent  to this family is well-defined on the whole surface $Y$. This implies that the holonomy of $\Psi$ along any closed loop on $Y$  equal to $\pm 1$.  
Assume first that the holonomy group of $\Psi$ is trivial. Then there exists a $1$-form $\omega $ on $Y$ such that $\Psi=\omega^k$. (Indeed, in this situation every branch of $\root k \of \Psi$ is a well-defined $1$-form on $Y$). In the remaining case the holonomy group of $\Psi$ is $\{\pm 1\}$ which implies that $k$ branches split into $\frac{k}{2}$ pairs where the two branches in each pair are analytic continuation of each other. Moreover the two branches in each pair coincide with the two branches of $\sqrt q$ for some quadratic differential $q$. Additionally, $\Psi=q^{\frac{k}{2}}$ independently of the choice of a pair. 
\end{proof}

\begin{defi+} For a $k$-differential $\Psi=f(z) dz^k,$ we call by its {\it dual} the differential $\Psi^*=i\cdot f(z) dz^k$, $i=\sqrt{-1}$. Obviously, $\Psi^*$ is globally well-defined  and attains purely imaginary values whenever $\Psi$ attains real values and vice versa. Additionally, $(\Psi^*)^*=-\Psi.$ 
\end{defi+}

\subsection{Introducing quasi-Strebel structures} 

The following definitions are inspired by the approach to Strebel quadratic differentials developed in \cite{BaSh} and by our previous results obtained in \cite{MaSh, BR, HoSh, STT}. 

\begin{defi+} Given a $k$-differential $\Psi$ on $Y$,  by its \emph{finite broken  trajectory}  we mean a continuous curve on $Y$ consisting of finitely many segments of trajectories of $\Psi$. A broken trajectory which is a non-selfintersecting closed curve on $Y$ is called \emph{closed}. Points where a trajectory is broken are called  \emph {switching points}. Usual periodic  trajectories (which might exist for $k$-differentials, but  can not cover almost all $Y$, see Corollary~\ref{cor:quadr}) is a particular case of closed broken trajectories.  
\end{defi+}

\begin{defi+}
Given a $k$-differential $\Psi$, we say that  a continuous and piecewise smooth function $\Phi:Y\setminus Cr_\Psi\to \bR$ is a \emph{prelevel function} of $\Psi$ if 

\smallskip
\noindent (i) $\Phi$ is non-constant on any open subset of $Y$; 

\smallskip
\noindent (ii) 
 all   connected components of its level curves are broken trajectories of $\Psi$. 
 
 \smallskip
  We say that a prelevel function $\Phi$ is a \emph{level function} of $\Psi$ if its \emph {switching set} $S_{\Psi,\Phi}$ (i.e., the set of all its switching points) consists of finitely many piecewise smooth compact curves on $Y$. 
  
  We say that two level functions $\Phi_1$ and $\Phi_2$ of $\Psi$ are \emph {level-equivalent} if every connected component of any level curve of $\Phi_1$ is a connected component of some level curve of $\Phi_2$ and conversely.
\end{defi+}





\begin{defi+}\label{def:strebel} Given a meromorphic  $k$-differential $\Psi$ defined on a compact Riemann surface $Y$ and a level function $\Phi:Y\setminus Cr_\Psi \to \bR$, we say that  $\Phi$ defines {\it a quasi-Strebel structure} for $\Psi$ (or that a pair $(\Psi,\Phi)$ is quasi-Strebel) if  

\smallskip
\noindent
(i)  $\Phi$ is a piecewise linear continuous function  consisting of  finitely many (local) linear functions in the charts of the flat model of $Y$;

\smallskip
\noindent
(ii) at each point where $\Psi$ is smooth  its differential  coincides with one of  the $k$ branches of $\text{Im }(\root k \of { \Psi})$;

\smallskip
\noindent
(iii)  almost all connected components of almost all level curves of $\Phi$ are closed broken trajectories of $\Psi$ and their  canonical lengths  are bounded from above.
\end{defi+}

\medskip
Level functions appearing in the Definition~\ref{def:strebel}  are called {\it piecewise standard} or {\it ps-level}  functions for short. Obviously the switching set of any ps-level function $\Phi$ coincides with the subset of $Y$ where $\Phi$ is not smooth and it consists of a finite collection of straight segments  on the flat model.  More exactly, the following result holds. 

\begin{lemma}\label{lm:directions}
The switching set $S_{\Psi,\Phi}$ of a ps-level function $\Phi$ of a quasi-Strebel structure $(\Psi,\Phi)$ consists of finitely many segments of trajectories of $\Psi$ and its dual $\Psi^*$
\end{lemma}

\begin{proof} Indeed consider a generic point $p$ of $S_{\Psi,\Phi}$. By the properties of the differential of $\Phi$, in the flat model of $\Psi$, $p$ lies on the line defined by the equation $l_1(W)=l_2(W)$, where $W$ is a local canonical coordinate   centered at $p$ and $l_1(W):=\text{Im}\; \left(\zeta_1 W\right)$ and $l_2(W):=\text{Im}\; \left(\zeta_2  W\right)$ are the linear functions with $\zeta_1^k=\zeta_2^k=1$. Any such line goes in the direction determined by the bisector of two $k$-th roots of unity and these are exactly the directions tangent to the trajectories of either $\Psi$ or $\Psi^\ast$ at the point $p$. 
\end{proof} 

\begin{remark}  Recall that given a $k$-differential $\Psi$,  its two ps-level functions $\Phi_1$ and $\Phi_2$ are {\it level-equivalent} if they have coinciding connected components of their level sets.  Since we are only interested in the behavior of the family of closed broken trajectories obtained as  connected components of the level curves or $\Phi_1$ (resp. $\Phi_2$) we can identify the quasi-Strebel structures $(\Psi,\Phi_1)$ and   $(\Psi,\Phi_2)$.  Thus  denoting by $[ \Phi]$ the class of all ps-level functions equivalent to $\Phi$, from now on we will  say that a  {\it quasi-Strebel structure} is a pair $(\Psi,[ \Phi])$. 
\end{remark}

\begin{defi+}\label{def:imp} Given a meromorphic $k$-differential $\Psi$ as above, we say that it has \emph{ admissible singularities} if 

\noindent
(i) $\Psi$ has no poles of order smaller than  $-k$;

\noindent
 (ii) at every pole of order $-k$ of $\Psi$  the residue should be equal to $i^k\cdot a$ where $a\in \bR$. 

\end{defi+} 

The motivation for the above definition is as follows.

\begin{lemma}\label{lm:poles}
If $\Psi$ is a $k$-differential admitting a quasi-Strebel structure $(\Psi,[ \Phi])$, then it has admissible singularities. 
\end{lemma}

\begin{proof} Indeed, if $\Psi$ has a pole of order smaller than or equal to  $-k$, there exists a neighborhood $V$ of this pole disjoint from the switching set $S_{\Psi,\Phi}$. If one of the level sets of $\Phi$ inside $V$ is a periodic trajectory, then the pole has to be of order $-k$ with the 
residue of the form  $i^k\cdot a$ where $a\in \bR$. Otherwise there exists a level set whose intersection with $V$ connects the boundary of $V$ with the pole and therefore is of infinite length - contradiction. 
\end{proof} 



 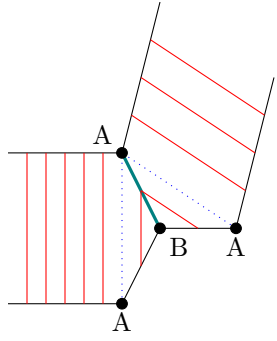
\begin{figure}



\begin{tikzpicture}
\draw (0,0) --(1.5,0);
\draw (0,2) --(1.5,2);
\draw [teal, very thick] (1.5,2) --(2,1);
\draw (2,1) --(3,1);
\draw (2,1) --(1.5,0);
\draw (1.5,2) --(2,4);
\draw (3,1) --(3.5,3);
\draw [blue, dotted] (1.5,0) --(1.5,2);
\draw [blue, dotted] (3,1) --(1.5,2);
\draw [red] (1.25,0) --(1.25,2);
\draw [red] (1,0) --(1,2);
\draw [red] (0.75,0) --(0.75,2);
\draw [red] (0.5,0) --(0.5,2);
\draw [red] (0.25,0) --(0.25,2);
\draw [red] (1.75,0.5) --(1.75,1.5);
\draw [red] (2.5,1) --(1.75,1.5);

\draw [red] (3.25,2) --(1.75,3);
\draw [red] (3.125,1.5) --(1.625,2.5);
\draw [red] (3.375,2.5) --(1.875,3.5);

\filldraw[black] (1.5,0) circle (2pt) node[anchor=north] {A};
\filldraw[black] (2,1) circle (2pt) node[anchor=north west] {B};
\filldraw[black] (3,1) circle (2pt) node[anchor=north] {A};
\filldraw[black] (1.5,2) circle (2pt) node[anchor=south east] {A};

\end{tikzpicture}

\caption {Cubic differential with one zero of order $2$ (marked by $A$), one pole of order $-2$ (marked by $B$) and two poles of order $-3$ (at the ends of two cylinders). The thick segment connecting $A$ and $B$ is the switching set of the quasi-Strebel structure whose broken trajectories are shown by the parallel lines. Finally, the dashed  lines are the special broken trajectories separating the cylinders.}
\label{figRes}
\end{figure}

\begin{defi+}
Given a ps-level function $\Phi$ for a  differential $\Psi$, we say that  a point $p$ in the switching set $Sw_{\Phi}$ is its {\it node}  if  an arbitrary small neighborhood of $p$   in $Sw_{\Phi}$ is not a straight segment without boundary. (In other words, a node is a vertex of $Sw_{\Phi}$  considered as a graph embedded in $Y$.)  In the above notation, connected components of the level curves of $\Phi$ containing nodes 
will be called {\it special}. 
\end{defi+}

One type of nodes is described by the following lemma. 

\begin{lemma}\label{lm:skeleton}
For any quasi-Strebel structure $(\Psi, \Phi)$, each vertex of valency one of  $ Sw_{\Phi}$ (which is considered as an embedded graph on $Y$) is a singular point of $\Psi$ whose order  for $k$ is odd, is not divisible by $k$ and  for $k$  even, not divisible by $\frac{k}{2}$.
Moreover, every singular point of $\Psi$ whose order is not an integer multiple of $\frac{k}{2}$ is a node of $ Sw_{\Phi}$. 
\end{lemma}

(Obviously, vertices of valency one are nodes of $Sw_{\Phi}$). 

\begin{proof}�
Around a vertex of valency one of the switching set, the level curves are formed by some segments of  smooth trajectories (possibly none) and one broken trajectory with a switching point of angle $\theta$, $0 < \theta < \pi$. Connecting these segments of trajectories by half circles, we get a loop around the vertex whose holonomy does not belong to $\{\pm 1\}$. Therefore, the vertex is a conical singularity whose total angle is not a multiple of $\pi$. Therefore its order is neither divisible by $k$ nor by $\frac{k}{2}$ for $k$  even.

Now if a conical singularity does not belong to $ Sw_{\Phi}$, then the level curves around it are formed only by some segments of  smooth trajectories. Connecting as above these segments of trajectories by half circles, we get a loop around the vertex whose holonomy  belongs  to $\{\pm 1\}$.      
Thus the order of the conical singularity is  an integer multiple of $\frac{k}{2}$.
\end{proof}

\begin{rema+} 
Notice that 
 for $k>2$,  there usually exist nodes which are not the singular points of $\Psi$, comp. Fig.~\ref{fig22}. We will call them {\it secondary singularities.}  �(These nodes are similar to additional turning points appearing in the asymptotic theory of linear ODE of order exceeding 2, see \cite{AKT}.) 
\end{rema+}

\subsection{Main theorem}\label{sec:general}


\begin{theorem}\label{th:main}
{\rm (i)}\,�For even $k>2$, every $k$-differential $\Psi$ with admissible singularities has a quasi-Strebel structure.

\noindent
{\rm (ii)}\,�For odd $k>2$,  any  $k$-differential $\Psi$ with admissible singularities  such that, up to a common factor, the period of $\root k \of \Psi$ along every path connecting any two conical singularities of $\Psi$  belongs  to $\QQ[e^{\frac{2i\pi}{k}}]$ has a quasi-Strebel structure.
\end{theorem}

\begin{remark} For $k$ odd, differentials satisfying  condition (ii) are dense in the space of all differentials with admissible singularities. 
\end{remark}











To prove Theorem~\ref{th:main}, �we will need the following preliminary statement. 

\begin{proposition}\label{pr:decomposition}
For any $k\ge 3$ and any $k$-differential $\Psi$ with admissible singularities, there exists a decomposition of the flat model of $\Psi$ into triangles, trapezoids and infinite cylinders such that

\noindent
{\rm (a)} each singularity of $\Psi$  is one of its vertices; 
 
 \noindent
{\rm (b)} each boundary edge of the decomposition is a segment of a horizontal trajectory.
\end{proposition}

\begin{proof}
To start with, observe that if $\Psi$ has a pole of order $-k$, then by Lemma~\ref{lm:poles} its residue is such that the corresponding infinite cylinder is bounded by a union of some critical trajectories and therefore it can be cut off from $Y$ along this union. After cutting off all such cylinders, we get a new surface $\widetilde Y$ with finite area whose boundary consists of horizontal trajectories. (Recall that there exists only  finitely  many critical trajectories in $Y$ since $\Psi$ has a finite number of singularities.) Consider all the remaining critical trajectories satisfying the condition that they are nowhere dense in $\widetilde Y$ and remove their union from $\widetilde Y$.  All the remaining critical trajectories are dense and some of them are self-intersecting while the other are not.  
Let us additionally cut $\widetilde Y$ along a certain part of every remaining critical trajectory having  self-intersections. Namely, we remove a fragment of such trajectory starting at the conical singularity and ending at its first self-intersection. (If such a trajectory both starts and ends at conical singularities we do the latter procedure for both of its endpoints). Denote by $CUT$ the constructed cut of $Y$ which includes every curve we removed so far.

Denote by $\gamma$ one of the remaining critical trajectories. Then  $\gamma$ is dense in some domain of $Y$ and non-selfintersecting. At this point, there are two cases: either $\gamma$  hits $CUT$ or it is dense in some domain $D\subseteq (Y\setminus CUT)$. In the first case, similarly to the above, we add to $CUT$ the fragment of $\gamma$ between the conical singularity and its first intersection with $CUT$. 

Let us show by contradiction that the second case is impossible. Take some critical trajectory $\gamma$ which is dense in some domain $D^\prime \subseteq (Y\setminus CUT)$.  Then any other critical trajectory which starts in $D^\prime$ must be dense in some subdomain $\D^{\prime\prime}\subseteq D^\prime$. 

The trajectory $\gamma$ returns arbitrarily close to its starting point $A$ (which is a singular point of $\Psi$). Therefore, there is an angular sector at $A$ contained in $D^\prime$ and whose angle  is at least $\pi$, see Figure~\ref{figDensity}. Since $k>2$, there is still another non-selfintersecting critical trajectory $\alpha$ in $D^\prime$ starting at $A$ transversally to $\gamma$.  (Observe that this argument does not work for quadratic differentials which explains why quadratic differentials are not always quasi-Strebel.) 
 Notice that $\alpha$ should also be dense in some $\D^{\prime\prime}\subseteq D^\prime$  (because the geometry of $D^\prime$ is that of a quadratic differential, see Proposition~\ref{pr:density}). Notice that the case $D^{\prime\prime}=D^\prime$ is impossible since    $\alpha$ and $\gamma$ are transverse. Taking $\alpha$ instead of $\gamma$ we can iterate this step. However we get a contradiction after a finite number of steps since there is no infinite sequence of density domains strictly included in each other. 

In this way, the connected components of the complement to the constructed cut $CUT$ have corners with angles  smaller than or equal to $\pi$. Indeed, at each singularity the angle between any two neighboring critical trajectories equals $2\pi/k<\pi$. Moreover, for any regular point where (at least) two segments of $CUT$ intersect, there are either (at least) four angles  smaller than $\pi$ ot there are three angles one of which is $\pi$. By the Gauss-Bonnet theorem, a flat surface such that every corner of its boundary has angle smaller than $\pi$ is either a convex polygon (if at least one corner has an angle strictly smaller than $\pi$) or a cylinder (if every  angle of every corner is  $\pi$ or, equivalently, if there are no corners). Cylinders can be cut out along some horizontal directions in order to get a parallelogram.
Now, each convex polygon can be cut into a finite number of trapezoids and  triangles by taking  lines parallel to  one side of the polygon. 
\end{proof}

  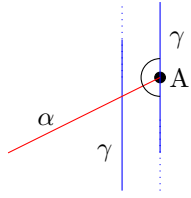
\begin{figure}

\begin{tikzpicture}
\filldraw[black] (2,2) circle (2pt) node[anchor=west] {A};
\draw [blue] (2,1) --(2,3);
\draw [blue, dotted] (2,0.5) --(2,1.5);
\draw [blue] (1.5,0.5) --(1.5,2.5);
\draw [blue, dotted] (1.5,2) --(1.5,3);
\draw [red] (2,2) --(0,1);

\filldraw[black] (0.5,1.25) circle (0pt) node[anchor=south] {$\alpha$};
\filldraw[black] (2,2.5) circle (0pt) node[anchor=west] {$\gamma$};
\filldraw[black] (1.5,1) circle (0pt) node[anchor=east] {$\gamma$};
\draw (2,2.25) arc (90:270:0.25) ;
\end{tikzpicture}


\vskip 0.5cm

\caption {The local picture at a critical point lying in a density domain.}
\label{figDensity}
\end{figure}

\begin{proof}[Proof of Theorem~\ref{th:main}] To settle Claim (i), 
we use Proposition~\ref{pr:decomposition} to decompose the flat surface $Y$ of $\Psi$ into infinite cylinders, triangles and trapezoids. Let us show that each of these tiles can be covered by a family of closed broken trajectories. Indeed, each infinite cylinder is covered by a family of smooth horizontal trajectories. 
Each triangle can be filled by the family of broken trajectories whose switching set consisting of three segments of bisectors meeting at the incenter,  i.e. the center of the inscribed circle, see the right part of Figure~\ref{figTriangle}. 
Finally, any trapezoid can be decomposed into a symmetric trapezoid and a triangle. Symmetric trapezoids can be covered by closed broken trajectories in a standard way shown in the left part of Figure~\ref{figTriangle}.

  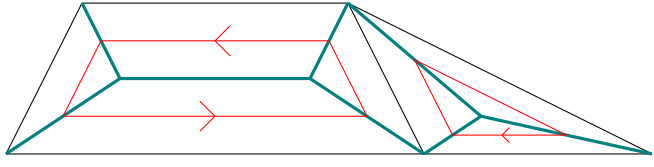
\begin{figure}



\begin{tikzpicture}
\draw (0,0) --(8.5,0);
\draw (0,0) --(1,2);
\draw (1,2) --(4.5,2);
\draw (8.5,0) --(4.5,2);
\draw (5.5,0) --(4.5,2);
\draw [teal, very thick] (0,0) --(1.5,1);
\draw [teal, very thick] (1,2) --(1.5,1);
\draw [teal, very thick] (4,1) --(1.5,1);
\draw [teal, very thick] (4,1) --(4.5,2);
\draw [teal, very thick] (4,1) --(5.5,0);

\draw [teal, very thick] (6.25,0.5) --(5.5,0);
\draw [teal, very thick] (6.25,0.5) --(8.5,0);
\draw [teal, very thick] (6.25,0.5) --(4.5,2);

\draw [red] (5.875,0.25) --(7.375,0.25);
\draw [red] (5.875,0.25) --(5.375,1.25);
\draw [red] (7.375,0.25) --(5.375,1.25);

\draw [red] (0.75,0.5) --(4.75,0.5);
\draw [red] (1.25,1.5) --(4.25,1.5);
\draw [red] (4.25,1.5) --(4.75,0.5);
\draw [red] (1.25,1.5) --(0.75,0.5);

\draw [red] (2.75,1.5) --(2.95,1.7);
\draw [red] (2.75,1.5) --(2.95,1.3);
\draw [red] (2.75,0.5) --(2.55,0.7);
\draw [red] (2.75,0.5) --(2.55,0.3);

\draw [red] (6.525,0.25) --(6.625,0.35);
\draw [red] (6.525,0.25) --(6.625,0.15);
\end{tikzpicture}

\caption {Tiling a general trapezoid by a symmetric trapezoid and a triangle and covering both tiles by closed broken trajectories.}
\label{figTriangle}
\end{figure}





To prove Claim (ii), we can without loss of generality assume that all the periods of $\root k \of \Psi$ lie in $\QQ[\zeta]$ where $\zeta$ is a $k$-th root of unity. Firstly,  define the subset $\mathcal L\subset Y$ consisting of all points $p$ such that a period of $\Psi$ over a path connecting $p$ of one of conical singularities is in $\QQ[\zeta]$. (Observe that by our assumptions if one period satisfies this property, then all of them will satisfy it as well.) We want to show that every vertex created in a decomposition from Proposition~\ref{pr:decomposition} belongs to $\mathcal L$. 

Indeed, choose two points  $A$ and $C$ in $\mathcal L$. Let $B$ be a point  in the intersection of the horizontal trajectories passing through $A$ and $C$ resp. We want to show that $B$ also belongs to $\mathcal L$. 
By choosing the appropriate $k$-th root we can assume that $\int_A^B\root k \of \Psi \in \bR$ and $\int_B^C\root k \of \Psi \in \zeta \cdot \bR$. Since $\QQ[\zeta]$ considered as a vector space over $\QQ$ is the direct sum of $\QQ[\zeta+\bar\zeta]\subset \bR$ and $\zeta\cdot \QQ[\zeta+\bar\zeta]\subset \zeta \cdot \bR$, the decomposition of the number $\int_A^C\root k \of \Psi $ into such summands is exactly $\int_A^C\root k \of \Psi=\int_A^B\root k \of \Psi+\int_B^C\root k \of \Psi$. Thus $\int_A^B\root k \of \Psi$ belongs to $\QQ[\zeta+\bar\zeta]\subset \QQ[\zeta]$ implying that $B\in \mathcal L$. 

\smallskip
Therefore the decomposition provided by Proposition~\ref{pr:decomposition} consists of triangles and trapezoids whose vertices belong to $\mathcal L$.
In Theorem 4 of \cite{Ke} R.~Kenyon proved that a convex polygon in $\bC$ is tileable with triangles whose angles are multiples of $\frac{\pi}{k}$ if and only if up to a common factor, its vertices lie in $\QQ[e^{\frac{2i\pi}{k}}]\subset \bC$. So, $\widetilde Y$ can be tiled with triangles whose sides are horizontal trajectories. Additionally, every such triangle can be filled by a family of broken trajectories with the switching set being the union of three bisectors, see the right part of Figure~\ref{figTriangle}.  
\end{proof}

\begin{remark} The result \cite{Ke}  of R.~Kenyon quoted above  states, in particular, that a convex polygon with sides in the directions of the $k$-th roots of unity cannot be tiled by triangles with sides in the same set of directions if the lengths of the sides of the polygon  do not belong to the field $\QQ[\cos (2\pi/k)]$. Because of this result we expect to find arithmetic obstructions to the existence of a quasi-Strebel structure for $k$-differentials with admissible singularities in the case of odd $k$.
\end{remark}

\section{Positive $k$-differentials and Heine-Stieltjes theory}\label{sec:HS}  

The material of this section is somewhat similar to our treatment of gradient and positive gradient quadratic differentials presented in \cite{BaSh}. 

\medskip
To connect quasi-Strebel structures with potential theory, let us recall some basic facts from complex analysis on Riemann surfaces, see e.g. \cite {GH}.
 Let~$Y$ be a Riemann surface (open or closed) and $h$ be a real- or
complex-valued smooth function on~$Y$.  

\begin{defi+}
The {\em Levy form} of $h$ (with respect to a local coordinate 
$z$) is a $(1,1)$-form given by 
\begin{equation}\label{levyn} 
\mu_h:=2i \frac{\partial^2 h}{\partial z \partial \bar z}\; dz\wedge d\bar z.
\end{equation}
\end{defi+}

In terms of the real and imaginary parts $(x,y)$ of  $z,$  $\mu_h$ is given by 
$$\mu_h=\left(  \frac{\d^2 h}{\d x^2}+ \frac{\d^2 h}{\d y^2} \right) \;  dx \wedge dy=\Delta h dx \wedge dy. $$

If $h$ is a smooth real-valued function, then $\mu_h$ is a  signed measure on $Y$ with a smooth density. In
potential theory $h$ is usually referred to as the {\em (logarithmic)
  potential}  of the measure $\mu_h$, see e.g. \cite{GH}, Ch.3. 
  
  Notice
that (\ref{levyn}) makes sense for  an arbitrary  complex-valued
distribution $h$ on~$Y$ if one interprets $\mu_h$  as a 2-current on
$Y$, i.e. a linear functional on the space of smooth compactly
supported functions on $Y$,  see e.g. \cite {Fed}. 
Such a current is necessarily exact since the inclusion of
smooth forms into currents induces the (co)homology isomorphism.
Notice that if $Y$ is compact and connected,
then the exactness of $\mu_h$ is equivalent to the vanishing of the
integral of $\mu_h$ over $Y$. 





\smallskip
Now observe that by definition, any ps-level function $\Phi$ for a $k$-differential $\Psi$ is piecewise-harmonic and continuous in the complement to the set of  poles of order $-k$ where it tends to $\pm \infty$ depending on the residue of $\Psi$ at such a pole, see Definition~\ref{def:imp}. Thus $\Phi$ defines a distribution on the Riemann surface $Y$ and its Levy form $\mu_\Phi$ is a measure supported on the union of the switching set $S_{\Psi,\Phi}$ and the set of poles of order $-k$. 

\medskip 
The final notion motivated by our original study of the Heine-Stieltjes theory is as follows. 

\begin{defi+}
A  $k$-differential $\Psi$   with at least one pole of order $-k$,  is called {\it positive} if there exists a  ps-level function $\Phi$  whose Levy form $\mu_{ \Phi}$ is a real measure the negative part of which is supported only on the set of poles of order $-k$ of $\Psi$.  
\end{defi+}




\medskip  
Let us now present some families of rational $k$-differentials which are positive in the above sense. (In these examples the first author initially  encountered quasi-Strebel structures and positive differentials of order exceeding $2$, see \cite {MaSh, STT, HoSh, BR}.) 

\medskip 

\begin{THEO}~\label{th:rat} Given an arbitrary monic polynomial $Q_k(z)=z^k+a_1z^{k-1}+...+a_k$ of degree $k$, consider the rational $k$-differential $\Psi=\frac{i^k dz^k}{Q_k(z)}, \; i=\sqrt{-1}$. Then $\Psi$ is a positive  differential.
\end{THEO}

Theorem~\ref{th:rat} is a special case of Corollary~\ref{cor:imp} below. 
Further, given a monic polynomial $Q(z)$ of degree $\ell$ and a positive integer $k\ge 2$, take the differential operator 
\begin{equation}\label{eq:op}
T=Q(z)\frac{d^k}{dz^k}.
\end{equation}

Assuming that $\ell \ge k$, consider the related Heine-Stieltjes problem which is as follows.  
For a non-negative integer $n$, find  all pairs of monic polynomials $(V(z), S(z))$ such that $\deg V=\ell-k$ and $\deg S=n$ which solve the equation 

\begin{equation}\label{eq:HS}
TS_n(z)-(n)_kV(z)S(z)=Q(z)\frac{d^k S}{dz^k}-(n)_kV(z)S(z)=0,
\end{equation}
 where $(n)_k:=n(n-1)\dots(n-k+1)$. ($V$ is classically known as the {\it Van-Vleck}  polynomial and $S$ as the {\it Stieltjes} polynomial for the equation \eqref{eq:HS}.)

\smallskip 
  The following results can be found in \cite{Sh}, Theorem 4. 

\begin{THEO} In the above notation, for every $n\ge k$, there exist exactly $\binom{n+\ell-k}{\ell -k}$ pairs $(V(z), S(z))$ counting multiplicities which solve \eqref{eq:HS}.  Moreover all roots of every $V(z)$ and every $S(z)$ are located in the convex hull $Conv(Q)\subset \bC$ of the roots of $Q(z)$. 
\end{THEO}

Given a non-negative integer $n$, denote by $\{V_{i,n}\}, \; i=1,\dots, \binom{n+\ell-k}{\ell -k}$ the set of all Van Vleck polynomials whose Stieltjes polynomials have degree $n$. 
 In the above notation, assume that $\{V_{i_n,n}\}_{n=k}^\infty$ is a sequence of Van Vleck polynomials converging to some monic polynomial  $\widetilde V(z)$. (Recall that all $V_{i,n}(z)$ have degree $k-\ell$ and their roots lie in $Conv(Q)$ which implies that there exist plenty of converging sequences $\{V_{i_n,n}\}_{n=N}^\infty$.)  Denote by $S_{i,n}$ the  Stieltjes polynomial corresponding  to $V_{i,n}$.  (Here $\deg S_{i,n}=n$.) 
 
 \begin{definition}
 The {\em root-counting measure} $\mu_P$ of a univariate polynomial $P=\prod_{j=1}^d (x-x_i)$ is, by definition,  given by 
 $$\mu_d:=\frac{1}{d}\sum_{j=1}^d \delta(x-x_i), $$
where $\delta(x-u)$ is the unit point mass located at $u$.  
\end{definition}

\begin{THEO}\label{th:C} In the above notation, the sequence of root-counting measure $\{\mu_{i_n,n}\}$ of $\{S_{i_n,n}\}_{n=k}^\infty$ converges to the probability measure $\mu:=\mu_{\widetilde V,Q}$ satisfying the condition that 
$$\C^k_\mu(z)=\frac{\widetilde V(z)}{Q(z)}$$ 
a.e. in the complex plane $\bC$, where $\C_\mu(z):=\int_{\bC}\frac{d\mu(\xi)}{z-\xi}$ is the Cauchy transform of $\mu$. 
\end{THEO} 

\begin{corollary}\label{cor:imp}
The $k$-th differential $\Psi:=\frac{i^k \cdot \widetilde V(z)}{Q(z)}dz^k$ is  positive. The corresponding ps-level function $\Phi$ is given by the logarithmic potential $u_\mu(z)$ of the probability measure $\mu$ where
$$u_\mu(z):=\int_\bC \log|z-\zeta| d\mu(\zeta).$$ 
\end{corollary}

\begin{proof} (Sketch, comp. the final section of \cite{HoSh} and \S~4 of \cite{BoSh}.) 
Recall that 
$$\C_\mu(z)=\frac{\partial u_\mu(z)}{\partial z}; \quad \mu=\frac{1}{\pi} \frac{\partial \C_\mu(z)} {\partial \bar z}$$
where the derivatives are understood in the distributional sense.  These relation imply that $\mu$ is the Levy form of $u(z)$. The logarithmic potential $u_\mu(z)$ is continuous outside the set of point masses of $\mu$ and harmonic outside its support. One can directly check that Theorem~\ref{th:C}  can be reinterpreted as the statement that the level curves of the logarithmic potential $u_\mu(z)$ of the measure $\mu$  constructed in this theorem 
 consist of segments of horizontal trajectories of the differential $\Psi(z)=\frac{i^k \widetilde V(z)}{Q(z)}dz^k$. Indeed, the gradient of $u_\mu(z)$ equals $\bar \C_\mu(z)$ where bar stands for the usual complex conjugation. Thus the tangent direction  $\delta\in \bC$  to the level curve of $u_\mu(z)$ at a point $p$ is given by  $\kappa\,i\, \bar \C_\mu(p)$ for some real $\kappa$. Thus it will satisfy the condition $\frac{i^k \widetilde V(p)}{Q(p)}\delta^k>0$. Additionally, $u_\mu(z)$ approaches $\log |z|$ when $|z|\to \infty$ which implies that almost all connected components of its level curves are closed non-selfintersecting curves in $\bC$. The measure $\mu$ is positive in $\bC$ and its extension to $\bC P^1\supset \bC$ is the exact $2$-current with the point mass $-1$ at $\infty$. Observe that $\infty$  is the pole of order $-k$ for the $k$-differential $\Psi(z)=\frac{i^k \widetilde V(p)}{Q(p)}dz^k$ with the residue $-i^k$. Finally, the support of $\mu$  �consists of a finite number of curve segments  globally forming an embedded forest in $\bC$ where each curve is either a segment of a horizontal trajectory of $\Psi(z)$ or a segment of a  trajectory of $\Psi^\ast(z)=\frac{i^{k+1} \widetilde V(z)}{Q(z)}dz^k$, see more details in \cite{HoSh}. 
\end{proof}

Concrete illustrations of the switching sets for  $u(z)$ of such cubic differentials with $\deg \widetilde V=1$ and $\deg Q=4$ are given in Fig.~\ref{figVV}, see more details and explanations in  Fig.~1 and 2 of \cite {HoSh}.

  \begin{figure}

\begin{center}
\includegraphics[scale=0.35]{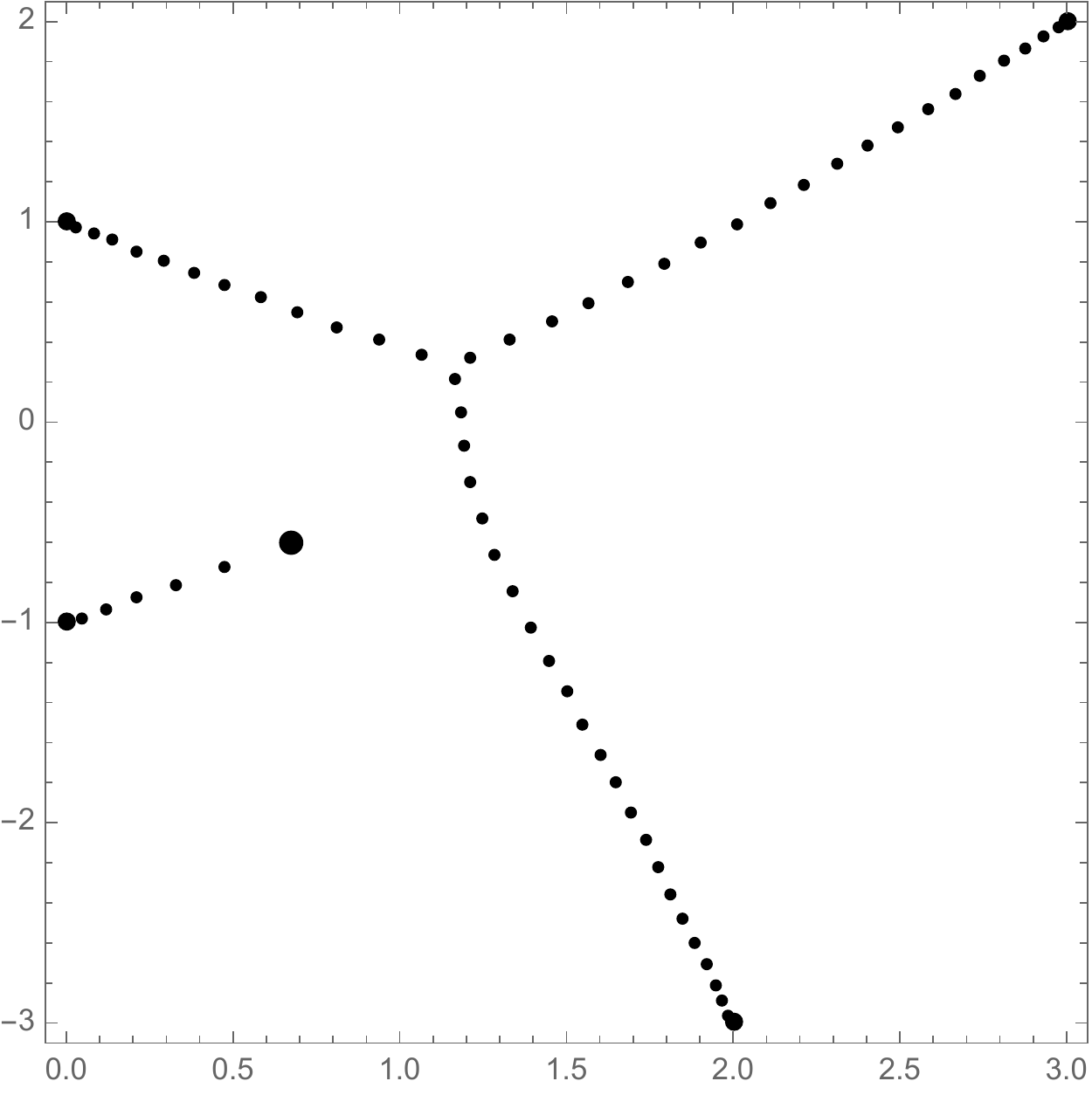} 
\includegraphics[scale=0.35]{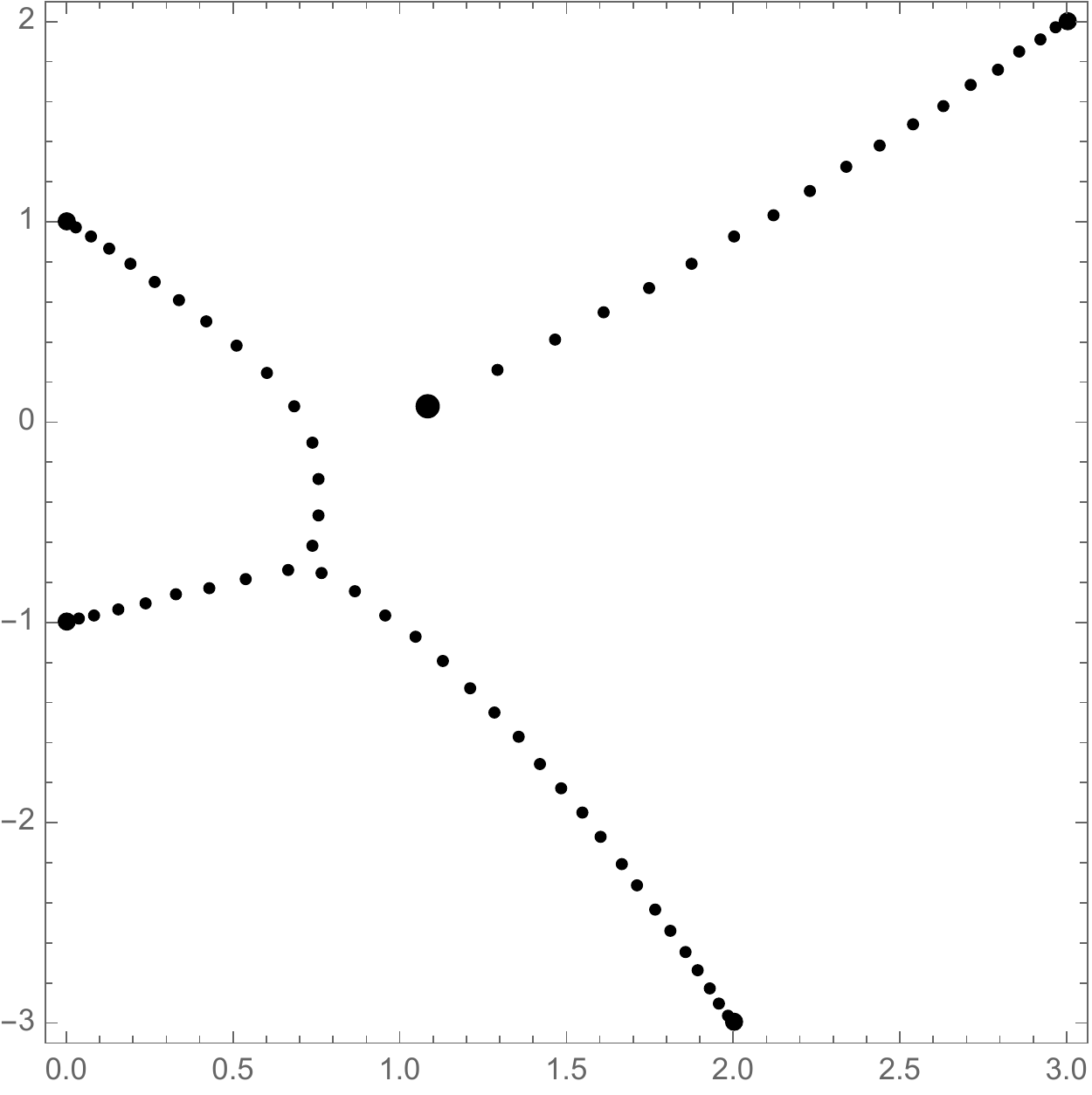} 
\end{center}

\vskip 0.5cm

\caption {The switching sets for the logarithmic potential $u_\mu(z)$ in case of the cubic differentials  $\Psi(z)=\frac{i^3 \cdot (z-\al)dz^3}{(z+i)(z-i)(z-2+3i)(z-3-2i)}$ with $\al\simeq 0.7-0.6i$ (left) and $\al\simeq 1.1+0.1i$ (right). (Possible choices of $\al$ are described in \cite{HoSh}).}
\label{figVV}
\end{figure}

\medskip
The special case when $\deg Q=k$ is called {\it exactly solvable} and is considered in more details in \cite{MaSh} and \cite{BR}. These papers contain substantial information about the switching sets of $u(z)$ for such differentials of order $k$ which are topologically planar trees in $\bC$ with vertices of valency one coinciding with all roots of $Q$ and special angles between the edges, see an illustration in Fig.~\ref{fig22}. 

  \begin{figure}

\begin{center}
\includegraphics[scale=0.6]{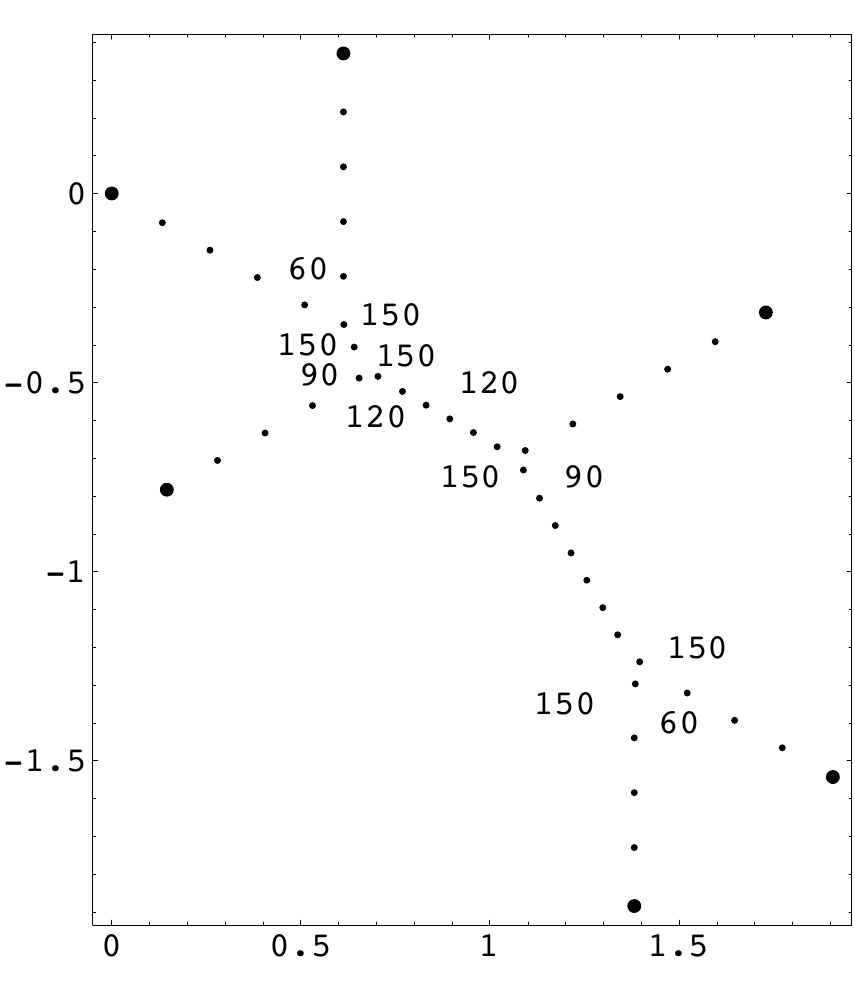} 
\end{center}

\vskip 0.5cm

\caption {The switching set of $u(z)$ for the rational sextic differential  $$\Psi(z)=\frac{- dz^6 }{z(z+i)(z-1-i)(z-2+2i)(z-2+3i)(z-3+2i)}$$ shown in a chart of the corresponding flat surface. (The numbers are the angles between the neighbouring  edges).}
\label{fig22}
\end{figure}

\section{Outlook}\label{sec:final}

\noindent 
{\bf 1.} Recall that the famous  Jenkins� Basic Structure Theorem \cite{Je}, Theorem 3.5, pp. 38-39], for any quadratic differential $\Psi$ given on a compact Riemann surface $Y$, up to a few exceptions, the set $Y\setminus  (K_\Psi \cup Cr_\Psi)$  consists of a finite number of the so-called circle, ring, strip and end domains. Here $K_\Psi$ is the union of all critical trajectories. (For the detailed definitions and information we refer to loc. cit). The names circle, ring and strip domain are reflecting the shapes of their images under the analytic continuation of the mapping given by the canonical coordinate; an end domain (also referred to as half-plane domain) is mapped by the canonical coordinate onto the half-plane.

 \begin{problem} Is there an analog of the latter theorem for higher order differentials?  \end{problem}
 
 Although due to  self-intersections of trajectories of $\Psi$ the existence of only finitely many types of  possible domains for $k>2$ seems problematic  one still wants  (at least crudely) to distinguish their possible global behaviour. 
 
 \smallskip
 \noindent 
{\bf 2.} The next question is related to Part (ii) of Theorem~\ref{th:main}.

\begin{problem}  Find necessary and sufficient conditions for a $k$-differential $\Psi$ of odd order $k$ to admit a quasi-Strebel structure? 
\end{problem}

Hopefully it is at least possible to find some arithmetic obstructions to the existence of a quasi-Strebel structure for odd $k$.

\smallskip
\noindent 
{\bf 3.} Typically a $k$-differential $\Psi$ with admissible singularities allows for an abundance of non-equivalent quasi-Strebel structures in the sense that  collections of their closed broken trajectories are distinct.  To streamline the situation, one can  introduce a natural partial order on  quasi-Strebel structures of a given $\Psi$. 
Observe that  the complement in $Y$ to the union of all special connected components of level curves and poles of order $-k$  is  a disjoint union of finitely many topological cylinders.  Each of them is covered by closed broken trajectories of the same type. We will call  such topological cylinders  \emph{packs of broken trajectories}. 

 In Figure~\ref{fig:many} we present a simple example of  two nonequivalent ps-level functions (quasi-Strebel structures) for the same quartic differential.  Observe additionally that on the right picture in Figure~\ref{fig:many}, the switching set consists of three connected components one of which is non-contractible. On the left picture of Figure~\ref{fig:many}, the switching set does not contain the zero of order 12.  
 




\smallskip�
Given two ps-level functions $\Phi_1$ and $\Phi_2$ on $Y$, we say that $\Phi_1$ gives a \emph{coarser subdivision of $Y$} �than $\Phi_2$ if the closure of every pack of broken trajectories of $\Phi_1$ is the union of the closures of some packs of broken trajectories of $\Phi_2$.  Notation, $\Phi_1\succeq \Phi_2$. In this situation we will also say that $\Phi_2$ provides a \emph {finer subdivision of $Y$} than $\Phi_1$.  We say that a quasi-Strebel structure $(\Psi, \Phi)$ is \emph{maximal} if there is no coarser ps-level function, i.e., no $\widetilde \Phi$ such that $\widetilde \Phi \succ \Phi$ exists. 





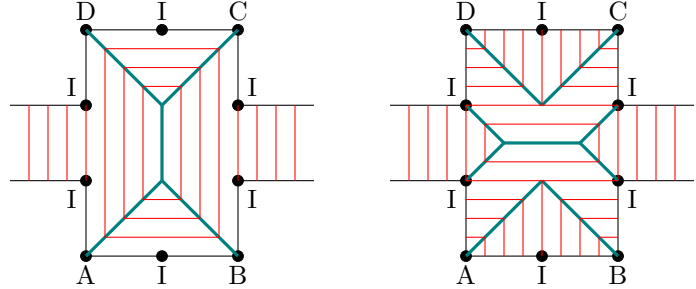
\begin{figure}

\begin{tikzpicture}
\draw (1,0) --(3,0);
\draw (1,3) --(3,3);
\draw (1,0)--(1,1);
\draw (0,1)--(1,1);
\draw (3,0)--(3,1);
\draw (3,1)--(4,1);
\draw (0,2)--(1,2);
\draw (3,2)--(4,2);
\draw (1,2)--(1,3);
\draw (3,2)--(3,3);

\draw (6,0) --(8,0);
\draw (6,3) --(8,3);
\draw (6,0)--(6,1);
\draw (5,1)--(6,1);
\draw (8,0)--(8,1);
\draw (8,1)--(9,1);
\draw (5,2)--(6,2);
\draw (8,2)--(9,2);
\draw (6,2)--(6,3);
\draw (8,2)--(8,3);

\filldraw[black] (1,0) circle (2pt) node[anchor=north] {A};
\filldraw[black] (3,0) circle (2pt) node[anchor=north] {B};
\filldraw[black] (1,3) circle (2pt) node[anchor=south] {D};
\filldraw[black] (3,3) circle (2pt) node[anchor=south] {C};

\filldraw[black] (6,0) circle (2pt) node[anchor=north] {A};
\filldraw[black] (8,0) circle (2pt) node[anchor=north] {B};
\filldraw[black] (6,3) circle (2pt) node[anchor=south] {D};
\filldraw[black] (8,3) circle (2pt) node[anchor=south] {C};

\filldraw[black] (2,0) circle (2pt) node[anchor=north] {I};
\filldraw[black] (2,3) circle (2pt) node[anchor=south] {I};
\filldraw[black] (1,1) circle (2pt) node[anchor=north east] {I};
\filldraw[black] (1,2) circle (2pt) node[anchor=south east] {I};
\filldraw[black] (3,1) circle (2pt) node[anchor=north west] {I};
\filldraw[black] (3,2) circle (2pt) node[anchor=south west] {I};

\filldraw[black] (7,0) circle (2pt) node[anchor=north] {I};
\filldraw[black] (7,3) circle (2pt) node[anchor=south] {I};
\filldraw[black] (6,1) circle (2pt) node[anchor=north east] {I};
\filldraw[black] (6,2) circle (2pt) node[anchor=south east] {I};
\filldraw[black] (8,1) circle (2pt) node[anchor=north west] {I};
\filldraw[black] (8,2) circle (2pt) node[anchor=south west] {I};

\draw [very thick] [teal] (1,0) -- (2,1);
\draw [very thick] [teal] (3,0) -- (2,1);
\draw [very thick] [teal] (1,3) -- (2,2);
\draw [very thick] [teal] (3,3) -- (2,2);
\draw [very thick] [teal] (2,1) -- (2,2);

\draw [very thick] [teal] (6,0) -- (7,1);
\draw [very thick] [teal] (8,0) -- (7,1);
\draw [very thick] [teal] (6,3) -- (7,2);
\draw [very thick] [teal] (8,3) -- (7,2);
\draw [very thick] [teal] (6.5,1.5) -- (7.5,1.5);
\draw [very thick] [teal] (6.5,1.5) -- (6,1);
\draw [very thick] [teal] (6.5,1.5) -- (6,2);
\draw [very thick] [teal] (7.5,1.5) -- (8,1);
\draw [very thick] [teal] (7.5,1.5) -- (8,2);

\draw [red] (0.25,1) --(0.25,2);
\draw [red] (0.5,1) --(0.5,2);
\draw [red] (0.75,1) --(0.75,2);
\draw [red] (1,1) --(1,2);
\draw [red] (1.25,0.25) --(1.25,2.75);
\draw [red] (1.5,0.5) --(1.5,2.5);
\draw [red] (1.75,0.75) --(1.75,2.25);
\draw [red] (2.25,0.75) --(2.25,2.25);
\draw [red] (2.5,0.5) --(2.5,2.5);
\draw [red] (2.75,0.25) --(2.75,2.75);
\draw [red] (3,1) --(3,2);
\draw [red] (3.25,1) --(3.25,2);
\draw [red] (3.5,1) --(3.5,2);
\draw [red] (3.75,1) --(3.75,2);

\draw [red] (1.25,0.25) --(2.75,0.25);
\draw [red] (1.25,2.75) --(2.75,2.75);
\draw [red] (1.5,0.5) --(2.5,0.5);
\draw [red] (1.5,2.5) --(2.5,2.5);
\draw [red] (1.75,0.75) --(2.25,0.75);
\draw [red] (1.75,2.25) --(2.25,2.25);

\draw [red] (5.25,1) --(5.25,2);
\draw [red] (5.5,1) --(5.5,2);
\draw [red] (5.75,1) --(5.75,2);
\draw [red] (6,1) --(6,2);
\draw [red] (8,1) --(8,2);
\draw [red] (8.25,1) --(8.25,2);
\draw [red] (8.5,1) --(8.5,2);
\draw [red] (8.75,1) --(8.75,2);
\draw [red] (6.25,1.25) --(6.25,1.75);
\draw [red] (7.75,1.25) --(7.75,1.75);
\draw [red] (6,1) --(8,1);
\draw [red] (6,2) --(8,2);
\draw [red] (6.25,1.25) --(7.75,1.25);
\draw [red] (6.25,1.75) --(7.75,1.75);

\draw [red] (6,0.75) --(6.75,0.75);
\draw [red] (8,0.75) --(7.25,0.75);
\draw [red] (6,2.25) --(6.75,2.25);
\draw [red] (8,2.25) --(7.25,2.25);
\draw [red] (6,0.5) --(6.5,0.5);
\draw [red] (8,0.5) --(7.5,0.5);
\draw [red] (6,2.5) --(6.5,2.5);
\draw [red] (8,2.5) --(7.5,2.5);
\draw [red] (6,0.25) --(6.25,0.25);
\draw [red] (8,0.25) --(7.75,0.25);
\draw [red] (6,2.75) --(6.25,2.75);
\draw [red] (8,2.75) --(7.75,2.75);

\draw [red] (6.25,0) --(6.25,0.25);
\draw [red] (6.25,2.75) --(6.25,3);
\draw [red] (6.5,0) --(6.5,0.5);
\draw [red] (6.5,2.5) --(6.5,3);
\draw [red] (6.75,0) --(6.75,0.75);
\draw [red] (6.75,2.25) --(6.75,3);
\draw [red] (7.25,0) --(7.25,0.75);
\draw [red] (7.25,2.25) --(7.25,3);
\draw [red] (7.5,0) --(7.5,0.5);
\draw [red] (7.5,2.5) --(7.5,3);
\draw [red] (7.75,0) --(7.75,0.25);
\draw [red] (7.75,2.75) --(7.75,3);
\draw [red] (7,0) --(7,1);
\draw [red] (7,2) --(7,3);
\end{tikzpicture}

\caption {A quartic differential with zero of order $12$ (marked by $I$), four poles of order $-3$ (marked by $A$, $B$, $C$ and $D$) and two poles of order $-4$ together with two nonequivalent ps-level functions. (The broken trajectories are shown by  the thin lines while the switching set is shown by the thick ones.)}
\label{fig:many}
\end{figure}

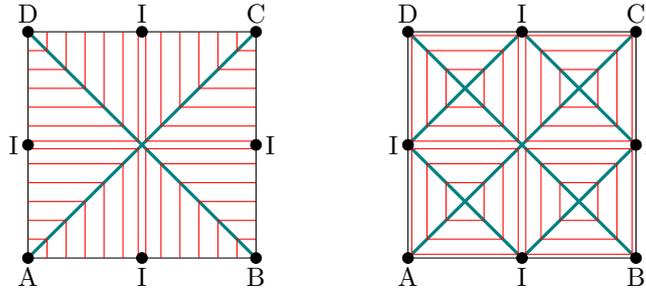
\begin{figure}

\begin{tikzpicture}
\draw (0,0) --(3,0);
\draw (0,0) --(0,3);
\draw (3,3)--(0,3);
\draw (3,3)--(3,0);

\draw (5,0) --(8,0);
\draw (5,0) --(5,3);
\draw (8,3)--(5,3);
\draw (8,3)--(8,0);

\draw [very thick] [teal] (0,0) -- (3,3);
\draw [very thick] [teal] (0,3) -- (3,0);

\draw [very thick] [teal] (5,0) -- (8,3);
\draw [very thick] [teal] (5,3) -- (8,0);
\draw [very thick] [teal] (5,1.5) -- (6.5,3);
\draw [very thick] [teal] (5,1.5) -- (6.5,0);
\draw [very thick] [teal] (8,1.5) -- (6.5,3);
\draw [very thick] [teal] (8,1.5) -- (6.5,0);

\draw [red] (1,0) --(1,1);
\draw [red](0,1) --(1,1);
\draw [red] (1.25,0) --(1.25,1.25);
\draw [red](0,1.25) --(1.25,1.25);
\draw [red] (0.5,0) --(0.5,0.5);
\draw [red](0,0.5) --(0.5,0.5);
\draw [red] (0.75,0) --(0.75,0.75);
\draw [red](0,0.75) --(0.75,0.75);
\draw [red] (0.25,0) --(0.25,0.25);
\draw [red](0,0.25) --(0.25,0.25);
\draw [red] (1.45,0) --(1.45,1.45);
\draw [red](0,1.45) --(1.45,1.45);

\draw [red] (2,3) --(2,2);
\draw [red](3,2) --(2,2);
\draw [red] (2.25,3) --(2.25,2.25);
\draw [red](3,2.25) --(2.25,2.25);
\draw [red] (2.5,3) --(2.5,2.5);
\draw [red](3,2.5) --(2.5,2.5);
\draw [red] (1.75,3) --(1.75,1.75);
\draw [red](3,1.75) --(1.75,1.75);
\draw [red] (2.75,3) --(2.75,2.75);
\draw [red](3,2.75) --(2.75,2.75);
\draw [red] (1.55,3) --(1.55,1.55);
\draw [red](3,1.55) --(1.55,1.55);

\draw [red] (0,1.75) --(1.25,1.75);
\draw [red] (1.25,3) --(1.25,1.75);
\draw [red] (0,2.75) --(0.25,2.75);
\draw [red] (0.25,3) --(0.25,2.75);
\draw [red] (0,2.5) --(0.5,2.5);
\draw [red] (0.5,3) --(0.5,2.5);
\draw [red] (0,2.25) --(0.75,2.25);
\draw [red] (0.75,3) --(0.75,2.25);
\draw [red] (0,2) --(1,2);
\draw [red] (1,3) --(1,2);
\draw [red] (0,1.55) --(1.45,1.55);
\draw [red](1.45,3) --(1.45,1.55);

\draw [red] (1.75,0) --(1.75,1.25);
\draw [red](3,1.25) --(1.75,1.25);
\draw [red] (2,0) --(2,1);
\draw [red](3,1) --(2,1);
\draw [red] (2.25,0) --(2.25,0.75);
\draw [red](3,0.75) --(2.25,0.75);
\draw [red] (2.5,0) --(2.5,0.5);
\draw [red](3,0.5) --(2.5,0.5);
\draw [red] (2.75,0) --(2.75,0.25);
\draw [red](3,0.25) --(2.75,0.25);
\draw [red] (1.55,0) --(1.55,1.45);
\draw [red](3,1.45) --(1.55,1.45);

\draw [red] (5.25,0.25) --(5.25,1.25);
\draw [red](6.25,0.25) --(5.25,0.25);
\draw [red] (5.25,1.25) --(6.25,1.25);
\draw [red](6.25,0.25) --(6.25,1.25);
\draw [red] (5.5,0.5) --(5.5,1);
\draw [red](6,1) --(5.5,1);
\draw [red] (6,1) --(6,0.5);
\draw [red](5.5,0.5) --(6,0.5);

\draw [red] (6.75,0.25) --(6.75,1.25);
\draw [red](7.75,0.25) --(6.75,0.25);
\draw [red] (6.75,1.25) --(7.75,1.25);
\draw [red](7.75,0.25) --(7.75,1.25);
\draw [red] (7,0.5) --(7,1);
\draw [red](7.5,1) --(7,1);
\draw [red] (7.5,1) --(7.5,0.5);
\draw [red](7,0.5) --(7.5,0.5);

\draw [red] (6.75,1.75) --(6.75,2.75);
\draw [red](7.75,1.75) --(6.75,1.75);
\draw [red] (6.75,2.75) --(7.75,2.75);
\draw [red](7.75,1.75) --(7.75,2.75);
\draw [red] (7,2) --(7,2.5);
\draw [red](7.5,2.5) --(7,2.5);
\draw [red] (7.5,2.5) --(7.5,2);
\draw [red](7,2) --(7.5,2);

\draw [red] (5.25,1.75) --(5.25,2.75);
\draw [red](6.25,1.75) --(5.25,1.75);
\draw [red] (5.25,2.75) --(6.25,2.75);
\draw [red](6.25,1.75) --(6.25,2.75);
\draw [red] (5.5,2) --(5.5,2.5);
\draw [red](6,2.5) --(5.5,2.5);
\draw [red] (6,2.5) --(6,2);
\draw [red](5.5,2) --(6,2);

\draw [red](5.05,0.05) --(6.45,0.05);
\draw [red](5.05,1.45) --(6.45,1.45);
\draw [red](5.05,0.05) --(5.05,1.45);
\draw [red](6.45,0.05) --(6.45,1.45);

\draw [red](6.55,0.05) --(7.95,0.05);
\draw [red](6.55,1.45) --(7.95,1.45);
\draw [red](6.55,0.05) --(6.55,1.45);
\draw [red](7.95,0.05) --(7.95,1.45);

\draw [red](5.05,1.55) --(6.45,1.55);
\draw [red](5.05,2.95) --(6.45,2.95);
\draw [red](5.05,1.55) --(5.05,2.95);
\draw [red](6.45,1.55) --(6.45,2.95);

\draw [red](6.55,1.55) --(7.95,1.55);
\draw [red](6.55,2.95) --(7.95,2.95);
\draw [red](6.55,1.55) --(6.55,2.95);
\draw [red](7.95,1.55) --(7.95,2.95);

\filldraw[black] (0,0) circle (2pt) node[anchor=north] {A};
\filldraw[black] (3,0) circle (2pt) node[anchor=north] {B};
\filldraw[black] (5,0) circle (2pt) node[anchor=north] {A};
\filldraw[black] (8,0) circle (2pt) node[anchor=north] {B};

\filldraw[black] (0,3) circle (2pt) node[anchor=south] {D};
\filldraw[black] (3,3) circle (2pt) node[anchor=south] {C};
\filldraw[black] (5,3) circle (2pt) node[anchor=south] {D};
\filldraw[black] (8,3) circle (2pt) node[anchor=south] {C};

\filldraw[black] (1.5,0) circle (2pt) node[anchor=north] {I};
\filldraw[black] (6.5,0) circle (2pt) node[anchor=north] {I};
\filldraw[black] (1.5,3) circle (2pt) node[anchor=south] {I};
\filldraw[black] (6.5,3) circle (2pt) node[anchor=south] {I};

\filldraw[black] (0,1.5) circle (2pt) node[anchor=east] {I};
\filldraw[black] (5,1.5) circle (2pt) node[anchor=east] {I};
\filldraw[black] (3,1.5) circle (2pt) node[anchor=west] {I};
\filldraw[black] (8,1.5) circle (2pt) node[anchor=west] {I};

\end{tikzpicture}

\caption {A quartic differential with a zero of order $4$ (marked by $I$) and four poles of order $-3$ (marked by $A$, $B$, $C$ and $D$) together with two nonequivalent ps-level functions. }
\label{fig:several}
\end{figure}
 
 \begin{problem}\label{pr:main} Given a $k$-differential $\Psi$ with admissible singularities, describe all its maximal quasi-Strebel structures.  
\end{problem}



\begin{thebibliography}{30}

\bibitem {AKT}  T.~Aoki,   T.~Kawai, and  Y.~Takei,  New turning points in the exact WKB analysis for higher order ordinary diffrential equations, Analyse alg\'ebrique des perturbations singuli\'eres. I, Hermann, (1994), 69--84.


\bibitem{BCGGM} M.~Bainbridge, D.~Chen, Q.~Gendron, S.~Grushevsky, M.~Moeller, 
Strata of k-differentials, Algebraic Geometry 6 (2) (2019) 196 -- 233.

\bibitem{BaSh} Y.~Baryshnikov, B.~Shapiro, Quadratic  differentials and  signed measures,  Journal d'Analyse mathematique, to appear. 

\bibitem{BH} Y.~Benoist, D.~Hulin,  Cubic differentials and hyperbolic convex sets. J. Differential Geom. 98 (2014), no. 1, 1--19. 

 \bibitem {BoSh} R.~B\o gvad, B.~Shapiro,  On mother body measures with algebraic Cauchy transform, L'Enseignement Math., vol. 62, (2016) 117--142. 

\bibitem{Bo} C.~Boissy, Connected components of the strata of the moduli space of meromorphic differentials. Commentarii Mathematici Helvetici, European Mathematical Society, 2015, 90 (2).

\bibitem{BR} T.~Bergkvist, H.~Rullg\aa rd, { On polynomial
eigenfunctions for a class of differential operators},
             Math. Res. Lett. {\bf 9} (2002), 153--171.
             
             \bibitem{Bu} A.~Bud,  
The image in the moduli space of curves of strata of meromorphic and quadratic differentials, arXiv:2002.04321. 
             
             



\bibitem{DW} D.~Dumas and M.~Wolf,  Polynomial cubic differentials and convex polygons in the  projective plane, Geom. Funct. Anal. Vol. 25 (2015) 1734--1798 

\bibitem {Fed} H.~Federer, Geometric measure theory, Die Grundlehren der mathematischen Wissenschaften, Band 153 Springer-Verlag New York Inc., New York 1969 xiv+676  




\bibitem{GH} P.~Griffiths, J.~Harris, {Principles of Algebraic Geometry,} A Wiley-Interscience Publ., vol.1., 1978. 



\bibitem{HoSh} T.~Holst, and B.~Shapiro,  On higher Heine-Stieltjes polynomials,  Isr. J. Math. 183 (2011) 321--347. 



\bibitem{HLL} Zh.~Huang, J.~Loftin, M.~Lucia,  Holomorphic cubic differentials and minimal Lagrangian surfaces in $CH2$. Math. Res. Lett. 20 (2013), no. 3, 501--520.

\bibitem{Ke} R.~Kenyon, Tilings of convex polygons, Annales de l'Institut Fourier, tome 47, no 3 (1997), p. 929--944. 


\bibitem{Je} J.~A.~Jenkins,
{ Univalent functions and conformal mapping.} 
Ergebnisse der Mathematik und ihrer Grenzgebiete. Neue Folge, Heft 18. Reihe: Moderne Funktionentheorie Springer-Verlag, Berlin-G\"ottingen-Heidelberg 1958 vi+169 pp.

\bibitem{La}�F.~Labourie,  Flat projective structures on surfaces and cubic holomorphic differentials. Pure Appl. Math. Q. 3 (2007), no. 4, Special Issue: In honor of Grigory Margulis. Part 1, 1057--1099. 


 
 
 \bibitem {MaSh} G.~M\'asson and B.~Shapiro,  On polynomial eigenfunctions of a hypergeometric-type operator,  Experiment. Math. vol 10, issue 4 (2001) 609--618. 

\bibitem{PePi} J.~V.~Pereira,  L.~Pirio, 
An invitation to web geometry. 
From Abel's addition theorem to the algebraization of codimension one webs. IMPA Mathematical Publications, 27th Brazilian Mathematics Colloquium, Instituto Nacional de Matem\'atica Pura e Aplicada (IMPA), Rio de Janeiro, 2009. xii+245 pp. 




\bibitem{Sh}   B.~Shapiro,  Algebro-geometric aspects of Heine-Stieltjes theory,  J. London Math. Soc.  83(1) (2011) 36--56.

\bibitem{STT} B.~Shapiro, K.~Takemura, M.~Tater, {On spectral polynomials of the Heun equation. II},  Comm. Math. Phys. 311(2) (2012),  277--300. 



\bibitem {Str}  K.~Strebel, { Quadratic differentials.} Ergebnisse der Mathematik und ihrer Grenzgebiete (3) [Results in Mathematics and Related Areas (3)], 5. Springer-Verlag, Berlin, 1984. xii+184 pp.

\bibitem {Str2} K.~Strebel,
{On the metric $\vert f(z)\vert \sp{\lambda }\vert dz\vert $ with
holomorphic $f$.}
Complex analysis (Warsaw, 1979), 323--336, Banach Center Publ., 11,
PWN, Warsaw, 1983.








\bibitem{Ta} G.~Tahar, 
Counting saddle connections in flat surfaces with poles of higher order, Geometriae Dedicata volume 196, pages  145 -- 186 (2018). 





\bibitem{Zo} A. Zorich. Flat Surfaces. Frontiers in Physics, Number Theory and Geometry, 439-586, 2006.



\end{thebibliography}
\end{document}